\documentclass[11pt,reqno]{amsart}
\usepackage[top=1.2in, bottom=1.2in, left=1.5in, right=1.5in]{geometry}

\usepackage{amsfonts, amssymb, amsmath, enumerate}

\numberwithin{equation}{section}

\DeclareMathOperator{\E}{\mathbb{E}}

\DeclareMathOperator*{\im}{Im}
\DeclareMathOperator*{\re}{Re}

\renewcommand{\Pr}[2][]{\mathbb{P}_{#1} \left\{ #2 \rule{0mm}{3mm}\right\}}

\def \R {\mathbb{R}}

\def \NN {\mathcal{N}}

\def \e {\varepsilon}
\def \d {\delta}
\def \l {\lambda}
\def \s {\sigma}

\def \HS {\mathrm{HS}}

\def \tran {\mathsf{T}}
\def \HS {\mathrm{HS}}
\def \< {\langle}
\def \> {\rangle}

\newtheorem{theorem}{Theorem}[section]

\newtheorem{corollary}[theorem]{Corollary}

\theoremstyle{remark}
\newtheorem{remark}[theorem]{Remark}

\title[]{Hanson-Wright inequality and sub-gaussian concentration}

\author{Mark Rudelson}
\author{Roman Vershynin}

\address{Department of Mathematics, University of Michigan, 530 Church St., Ann Arbor, MI 48109, U.S.A.}
\email{\{rudelson, romanv\}@umich.edu}
\thanks{M. R. was partially supported by NSF grant DMS 1161372. R. V. was partially supported by NSF grant DMS 1001829 and 1265782.}

\date{\today}

\begin{document}

\begin{abstract}
In this expository note, we give a modern proof of Hanson-Wright inequality for quadratic forms in
sub-gaussian random variables.
We deduce a useful concentration inequality for sub-gaussian random vectors.
Two examples are given to illustrate these results: a concentration of
distances between random vectors and subspaces, and a bound on the
norms of products of random and deterministic matrices.
\end{abstract}

\maketitle

\section{Hanson-Wright inequality}

Hanson-Wright inequality is a general concentration result for quadratic forms in
sub-gaussian random variables.
A version of this theorem was first proved in \cite{HW, W}, however with one weak point
mentioned in Remark~\ref{rem: original HW}.
In this article we give a modern proof of Hanson-Wright inequality,
which automatically fixes the original weak point.
We then deduce a useful concentration inequality for sub-gaussian random vectors, and illustrate
it with two applications.

Our arguments use standard tools of high-dimensional probability. The reader unfamiliar with them
may benefit from consulting the tutorial \cite{V RMT}. Still, we will recall the basic notions where possible.
A random variable $\xi$ is called sub-gaussian if its distribution is dominated by that of a normal random variable.
This can be expressed by requiring that $\E \exp(\xi^2/K^2) \le 2$ for some $K>0$; the infimum of
such $K$ is traditionally called the sub-gaussian
 or $\psi_2$
 norm of $\xi$.
 This turns the set of subgaussian random variables into the Orlicz space with the Orlicz function $\psi_2(t)=  \exp(t^2)-1$.
  A number of other equivalent definitions
are used in the literature. In particular, $\xi$ is sub-gaussian if an only if
$\E|\xi|^p = O(p)^{p/2}$ as $p \to \infty$, so we can redefine the sub-gaussian norm of $\xi$ as
$$
\|\xi\|_{\psi_2} = \sup_{p \ge 1} p^{-1/2} (\E|X|^p)^{1/p}.
$$
One can show that $\|\xi\|_{\psi_2}$  defined this way is within an absolute constant factor
from the infimum of $K>0$ mentioned above,
see \cite[Section 5.2.3]{V RMT}.
One can similarly define sub-exponential random variables, i.e. by requiring that
$\|\xi\|_{\psi_1} = \sup_{p \ge 1} p^{-1} (\E|X|^p)^{1/p} < \infty$.

For an $m \times n$ matrix $A = (a_{ij})$, recall that the operator norm of $A$
is $\|A\| = \max_{x \neq 0} \|Ax\|_2 / \|x\|_2$
and the Hilbert-Schmidt (or Frobenius) norm of $A$
is $\|A\|_\HS = ( \sum_{i,j} |a_{i,j}|^2 )^{1/2}$.
Throughout the paper, $C, C_1, c, c_1, \ldots$ denote positive absolute constants.

\begin{theorem}[Hanson-Wright inequality]					\label{thm: HW}
  Let $X = (X_1,\ldots,X_n) \in \R^n$ be a random vector with independent
  components $X_i$ which satisfy
  $\E X_i = 0$ and $\|X_i\|_{\psi_2} \le K$.
  Let $A$ be an $n \times n$ matrix. Then, for every $t \ge 0$,
  $$
  \Pr{|X^\tran A X - \E X^\tran A X| > t}
  \le 2 \exp \Big[ - c \min \Big( \frac{t^2}{K^4 \|A\|_\HS^2}, \frac{t}{K^2 \|A\|} \Big) \Big].
  $$
\end{theorem}

\begin{remark} [Related results]    \label{rem: original HW}
One of the aims of this note is to give a simple and self-contained proof of the Hanson--Wright inequality
using only the  standard toolkit of the large deviation theory.
Several partial results and alternative proofs are scattered in the literature.

  Improving upon an earlier result on Hanson-Wright \cite{HW}, Wright \cite{W} established
  a slightly weaker version of Theorem~\ref{thm: HW}.
  Instead of $\|A\| = \|(a_{ij})\|$, both papers had $\|( |a_{ij}| )\|$ in the right side.
  The latter norm can be much larger than the norm of $A$, and it is often less easy to compute.
  This weak point went unnoticed in several later applications of Hanson-Wright inequality,
  however it was clear to experts that it could be fixed.

A proof for the case where  $X_1,\ldots,X_n$ are independent symmetric Bernoulli random variables
appears in the lecture notes of Nelson \cite{N}.
The moment inequality which essentially implies the result of \cite{N} can be also found in \cite{DKN}.
A different approach to Hanson-Wright inequality, due to Rauhut and Tropp,
can be found in \cite[Proposition 8.13]{FR}. It is presented for diagonal-free matrices (however this
assumption can be removed by treating the diagonal separately as is done below), and
for independent symmetric Bernoulli random variables (but the proof can be extended to
sub-gaussian random variables).

An upper bound for $\Pr{X^\tran A X - \E X^\tran A X > t}$, which is equivalent to what appears
in the Hanson--Wright inequality, can be found in \cite{HKZ}. However, the assumptions in \cite{HKZ}
are somewhat different. On the one hand, it is assumed that the matrix $A$ is positive-semidefinite, while in our
result $A$ can be arbitrary. On the other hand, a weaker assumption is placed on the random vector $X = (X_1,\ldots,X_n)$. Instead of assuming that the coordinates of $X$ are independent subgaussian random variables,
it is assumed in \cite{HKZ} that the marginals of $X$ are uniformly subgaussian, i.e., that
$\sup_{y \in S^{n-1}} \ \|\langle X,y\rangle \|_{\psi_2} \le K$.

The paper \cite{BM} contains an alternative short proof of Hanson--Wright inequality due to Latala
for diagonal-free matrices.
Like in the proof below, Latala's argument uses decoupling of the order 2 chaos. However, unlike the current paper, which uses a simple decoupling argument of Bourgain \cite{B},  his proof uses a more general and
more difficult decoupling theorem for U-statistics due to de la Pe\~{n}a and Montgomery-Smith \cite{dlPM}. For an extensive discussion of modern decoupling methods see \cite{dlPG}.

Large deviation inequalities for polynomials of higher degree, which extend the Hanson-Wright type inequalities, have been obtained by Latala \cite{L} and recently by Adamczak and Wolff \cite{AW}.
\end{remark}

\begin{proof}[Proof of Theorem~\ref{thm: HW}]
By replacing $X$ with $X/K$ we can assume without loss of generality that $K=1$.
Let us first estimate
$$
p := \Pr{X^\tran A X - \E X^\tran A X > t}.
$$
Let $A = (a_{ij})_{i,j=1}^n$. By independence and zero mean of $X_i$, we can represent
\begin{align*}
X^\tran A X - \E X^\tran A X
&= \sum_{i,j} a_{ij} X_i X_j - \sum_i a_{ii} \E X_i^2 \\
&= \sum_i a_{ii} (X_i^2 - \E X_i^2) + \sum_{i,j:\, i \ne j} a_{ij} X_i X_j.
\end{align*}
The problem reduces to estimating the diagonal and off-diagonal sums:
$$
p \le \Pr{ \sum_i a_{ii} (X_i^2 - \E X_i^2) > t/2 } + \Pr{ \sum_{i,j:\, i \ne j} a_{ij} X_i X_j > t/2 }
=: p_1 + p_2.
$$

\medskip
{\bf Step 1: diagonal sum.} Note that $X_i^2 - \E X_i^2$ are independent mean-zero
sub-exponential random variables, and
$$
\|X_i^2 - \E X_i^2\|_{\psi_1} \le 2 \|X_i^2\|_{\psi_1} \le 4 \|X_i\|_{\psi_2}^2 \le 4K^2.
$$
These standard bounds can be found in \cite[Remark~5.18 and Lemma 5.14]{V RMT}.
Then we can use a Bernstein-type inequality (see \cite[Proposition~5.16]{V RMT}) and obtain
\begin{equation}				\label{eq: pone}
p_1 \le \Big[ - c \min \Big( \frac{t^2}{\sum_i a_{ii}^2}, \frac{t}{\max_i |a_{ii}|} \Big) \Big]
\le \exp \Big[ - c \min \Big( \frac{t^2}{\|A\|_\HS^2}, \frac{t}{\|A\|} \Big) \Big].
\end{equation}

\medskip
{\bf Step 2: decoupling.} It remains to bound the off-diagonal sum
$$
S := \sum_{i,j:\, i \ne j} a_{ij} X_i X_j.
$$
The argument will be based on estimating the moment generating function of $S$ by
decoupling and reduction to normal random variables.

Let $\l>0$ be a parameter whose value we will determine later.
By Chebyshev's inequality, we have
\begin{equation}				\label{eq: Chebyshev}
p_2 = \Pr{S > t/2} = \Pr{\l S > \l t/2} \le \exp(-\l t/2) \E \exp(\l S).
\end{equation}
Consider independent Bernoulli random variables $\d_i \in \{0,1\}$ with $\E \d_i=1/2$.
Since $\E \d_i (1-\d_j)$ equals $1/4$ for $i \ne j$ and $0$ for $i=j$, we have
$$
S = 4 \E_\d S_\d, \quad \text{where} \quad S_\d = \sum_{i,j} \d_i (1-\d_j) a_{ij} X_i X_j.
$$
Here $\E_\d$ denotes the expectation with respect to $\d=(\d_1,\ldots,\d_n)$.
Jensen's inequality yields
\begin{equation}				\label{eq: MGF S}
\E \exp(\l S) \le \E_{X,\d} \exp(4 \l S_\d)
\end{equation}
where $E_{X,\d}$ denotes expectation with respect to both $X$ and $\d$.
Consider the set of indices $\Lambda_\d = \{ i \in [n] :\, \d_i = 1 \}$ and express
$$
S_\d = \sum_{i \in \Lambda_\d, \, j \in \Lambda_\d^c} a_{ij} X_i X_j
= \sum_{j \in \Lambda_\d^c} X_j \Big( \sum_{i \in \Lambda_\d} a_{ij} X_i \Big).
$$

Now we condition on $\d$ and $(X_i)_{i \in \Lambda_\d}$.
Then $S_\d$ is a linear combination of mean-zero sub-gaussian
random variables $X_j$, $j \in \Lambda_\d^c$, with fixed coefficients.
It follows that the conditional distribution of $S_\d$ is sub-gaussian, and its sub-gaussian norm
is bounded by the $\ell_2$-norm of the coefficient vector (see e.g. in \cite[Lemma~5.9]{V RMT}).
Specifically,
$$
\|S_\d\|_{\psi_2} \le C \s_\d \quad \text{where} \quad
\s_\d^2 := \sum_{j \in \Lambda_\d^c} \Big( \sum_{i \in \Lambda_\d} a_{ij} X_i \Big)^2.
$$
Next, we use a standard estimate of the moment generating function of centered sub-gaussian random variables,
see \cite[Lemma~5.5]{V RMT}. It yields
$$
\E_{(X_j)_{j \in \Lambda_\d^c}} \exp(4 \l S_\d) \le \exp(C \l^2 \|S_\d\|_{\psi_2}^2) \le \exp(C' \l^2 \s_\d^2).
$$
Taking expectations of both sides with respect to $(X_i)_{i \in \Lambda_\d}$, we obtain
\begin{equation}				\label{eq: E delta}
\E_X \exp(4 \l S_\d) \le \E_X \exp(C' \l^2 \s_\d^2) =: E_\d.
\end{equation}
Recall that this estimate holds for every fixed $\d$.
It remains to estimate $E_\d$.

\medskip
{\bf Step 3: reduction to normal random variables.}
Consider $g = (g_1,\ldots, g_n)$ where $g_i$ are independent $N(0,1)$ random variables.
The rotation invariance of normal distribution implies that for each fixed $\d$ and $X$,
we have
$$
Z := \sum_{j \in \Lambda_\d^c} g_j \Big( \sum_{i \in \Lambda_\d} a_{ij} X_i \Big)
\sim N(0,\s_\d^2).
$$
By the formula for the moment generating function of normal distribution,
we have $\E_g \exp(s Z) = \exp(s^2 \s_\d^2/2)$. Comparing this with the
formula defining $E_\d$ in \eqref{eq: E delta}, we find that the two expressions are somewhat similar.
Choosing $s^2=2C' \l^2$, we can match the two expressions as follows:
$$
E_\d
= \E_{X,g} \exp(C_1 \l Z)
$$
where $C_1 = \sqrt{2C'}$.

Rearranging the terms, we can write
$Z = \sum_{i \in \Lambda_\d} X_i \Big( \sum_{j \in \Lambda_\d^c} a_{ij} g_j \Big)$.
Then we can bound the moment generating function of $Z$ in the same
way we bounded the moment generating function of $S_\d$ in Step~2,
only now relying on the sub-gaussian properties of $X_i$, $i \in \Lambda_\d$.
We obtain
$$
E_\d
\le \E_g \exp \Big[ C_2 \l^2 \sum_{i \in \Lambda_\d} \Big( \sum_{j \in \Lambda_\d^c} a_{ij} g_j \Big)^2 \Big].
$$
To express this more compactly, let $P_\d$ denotes the coordinate projection
(restriction) of $\R^n$ onto $\R^{\Lambda_\d}$, and define the
matrix $A_\d = P_\d A (I-P_\d)$. Then what we obtained
$$
E_\d \le \E_g \exp \Big( C_2 \l^2 \|A_\d g\|_2^2 \Big).
$$
Recall that this bound holds for each fixed $\d$.
We have removed the original random variables $X_i$ from the problem,
so it now becomes a problem about normal random variables $g_i$.

\medskip
{\bf Step 4: calculation for normal random variables.}
By the rotation invariance of the distribution of $g$, the random variable $\|A_\d g\|_2^2$
is distributed identically with $\sum_i s_i^2 g_i^2$ where $s_i$ denote the singular values of $A_\d$.
Hence by independence,
$$
E_\d = \E_g \exp \Big( C_2 \l^2 \sum_i s_i^2 g_i^2 \Big)
= \prod_i \E_g \exp \big( C_2 \l^2 s_i^2 g_i^2 \big).
$$
Note that each $g_i^2$ has the chi-squared distribution with one degree of freedom,
whose moment generating function is $\E \exp(t g_i^2) = (1-2t)^{-1/2}$
for $t<1/2$. Therefore
$$
E_\d \le \prod_i \big( 1- 2 C_2 \l^2 s_i^2 \big)^{-1/2}
\quad \text{provided } \quad
\max_i C_2 \l^2 s_i^2 < 1/2.
$$
Using the numeric inequality $(1-z)^{-1/2} \le e^z$ which is valid for all $0 \le z \le 1/2$,
we can simplify this as follows:
$$
E_\d \le \prod_i \exp(C_3 \l^2 s_i^2)
= \exp \Big( C_3 \l^2 \sum_i s_i^2 \Big)
\quad \text{provided } \quad
\max_i C_3 \l^2 s_i^2 < 1/2.
$$
Since $\max_i s_i = \|A_\d\| \le \|A\|$ and $\sum_i s_i^2 = \|A_\d\|_\HS^2 \le \|A\|_\HS$,
we have proved the following:
$$
E_\d \le \exp \big( C_3 \l^2 \|A\|_\HS^2 \big) \quad \text{for } \l \le c_0 / \|A\|.
$$
This is a uniform bound for all $\d$. Now we take expectation with respect to $\d$.
Recalling \eqref{eq: MGF S} and \eqref{eq: E delta}, we obtain the following estimate
on the moment generating function of $S$:
$$
\E \exp(\l S)
\le \E_\d E_\d
\le \exp \big( C_3 \l^2 \|A\|_\HS^2 \big) \quad \text{for } \l \le c_0 / \|A\|.
$$

\medskip
{\bf Step 5: conclusion.}
Putting this estimate into the exponential Chebyshev's inequality \eqref{eq: Chebyshev},
we obtain
$$
p_2 \le \exp \big( -\l t/ 2 + C_3 \l^2 \|A\|_\HS^2 \big) \quad \text{for } \l \le c_0 / \|A\|.
$$
Optimizing over $\l$, we conclude that
$$
p_2 \le \exp \Big[ - c \min \Big( \frac{t^2}{\|A\|_\HS^2}, \frac{t}{\|A\|} \Big) \Big] =: p(A, t).
$$
Now we combine with a similar estimate \eqref{eq: pone} for $p_1$ and obtain
$$
p = p_1 + p_2 \le 2 p(A,t).
$$
Repeating the argument for $-A$ instead of $A$, we get
$\Pr{X^\tran A X - \E X^\tran A X < - t} \le 2 p(A,t)$. Combining the two events, we obtain
$\Pr{|X^\tran A X - \E X^\tran A X| > t} \le 4 p(A,t)$.
Finally, one can reduce the factor $4$ to $2$ by adjusting the constant $c$ in $p(A,t)$.
The proof is complete.
\end{proof}

\section{Sub-gaussian concentration}					\label{s: subgauss}

Hanson-Wright inequality has a useful consequence, a
concentration inequality for random vectors with independent
sub-gaussian coordinates.

\begin{theorem}[Sub-gaussian concentration]				\label{thm: conc}
  Let $A$ be a fixed $m \times n$ matrix.
  Consider a random vector $X = (X_1,\ldots,X_n)$ where $X_i$ are
  independent random variables satisfying
  $\E X_i = 0$, $\E X_i^2 = 1$ and $\|X_i\|_{\psi_2} \le K$.
  Then for any $t \ge 0$, we have
  $$
  \Pr{ \big| \|AX\|_2 - \|A\|_\HS \big| > t }
  \le 2 \exp \Big( - \frac{c t^2}{K^4 \|A\|^2} \Big).
  $$
\end{theorem}

\begin{remark}
  The consequence of Theorem~\ref{thm: conc} can be alternatively
  formulated as follows: the random variable $Z = \|AX\|_2 - \|A\|_\HS$ is sub-gaussian,
  and $\|Z\|_{\psi_2} \le C K^2 \|A\|$.
\end{remark}

\begin{remark}
  A few special cases of Theorem~\ref{thm: conc} can be easily deduced from
  classical concentration inequalities.
  For Gaussian random variables $X_i$, this result is a standard
  consequence of Gaussian concentration, see e.g. \cite{Ledoux concentration}.
  For bounded random variables $X_i$, it can be deduced in a similar way
  from Talagrand's concentration for convex Lipschitz functions \cite{Tal concentration},
  see \cite[Theorem~2.1.13]{Tao RMT}.
  For more general random variables, one can find versions of Theorem~\ref{thm: conc}
  with varying degrees of generality scattered in the literature
  (e.g. the appendix of \cite{EYY}).
  However, we were unable to find Theorem~\ref{thm: conc} in the existing literature.
\end{remark}

\begin{proof}
Let us apply Hanson-Wright inequality, Theorem~\ref{thm: HW},
for the matrix $Q = A^\tran A$.
Since $X^\tran Q X = \|AX\|_2^2$, we have
$\E X^\tran Q X = \|A\|_\HS^2$. Also, note that since all $X_i$ have unit variance,
we have $K \ge 2^{-1/2}$. Thus we obtain for any $u \ge 0$ that
$$
\Pr{ \big| \|AX\|_2^2 - \|A\|_\HS^2 \big| > u }
\le 2 \exp \Big[ - \frac{C}{K^4} \min \Big( \frac{u}{\|A\|^2}, \frac{u^2}{\|A^\tran A\|_\HS^2} \Big) \Big].
$$
Let $\e \ge 0$ be arbitrary, and let us use this estimate for $u = \e \|A\|_\HS^2$.
Since $\|A^\tran A\|_\HS^2 \le \|A^\tran\|^2 \|A\|_\HS^2 = \|A\|^2 \|A\|_\HS^2$,
it follows that
\begin{equation}				\label{eq: AX squared}
\Pr{ \big| \|AX\|_2^2 - \|A\|_\HS^2 \big| > \e \|A\|_\HS^2 }
\le 2 \exp \Big[ - c \min(\e,\e^2) \, \frac{\|A\|_\HS^2}{K^4 \|A\|^2} \Big].
\end{equation}
Now let $\d \ge 0$ be arbitrary; we shall use this inequality for $\e = \max(\d,\d^2)$.
Observe that the (likely) event $\big| \|AX\|_2^2 - \|A\|_\HS^2 \big| \le \e \|A\|_\HS^2$
implies the event $\big| \|AX\|_2 - \|A\|_\HS \big| \le \d \|A\|_\HS$.
This can be seen by dividing both sides of the inequalities by $\|A\|_\HS^2$ and
$\|A\|_\HS$ respectively, and using the numeric bound
$\max(|z-1|, |z-1|^2) \le |z^2-1|$, which is valid for all $z \ge 0$.
Using this observation along with the identity $\min(\e,\e^2)=\d^2$,
we deduce from \eqref{eq: AX squared} that
$$
\Pr{ \big| \|AX\|_2 - \|A\|_\HS \big| > \d \|A\|_\HS }
\le 2 \exp \Big( - c \d^2 \, \frac{\|A\|_\HS^2}{K^4 \|A\|^2} \Big).
$$
Setting $\d = t/\|A\|_\HS$, we obtain the desired inequality.
\end{proof}

\subsection{Small ball probabilities}

Using a standard symmetrization argument, we can deduce
from Theorem~\ref{thm: conc} some bounds on small ball probabilities.
The following result is due to Latala et al. \cite[Theorem~2.5]{LMOT}.

\begin{corollary}[Small ball probabilities]		\label{cor: SBP}
  Let $A$ be a fixed $m \times n$ matrix.
  Consider a random vector $X = (X_1,\ldots,X_n)$ where $X_i$ are
  independent random variables satisfying
  $\E X_i = 0$, $\E X_i^2 = 1$ and $\|X_i\|_{\psi_2} \le K$.
  Then for every $y \in \R^m$ we have
  $$
  \Pr{ \|AX-y\|_2 < \frac{1}{2} \|A\|_\HS }
  \le 2 \exp \Big( - \frac{c \|A\|_\HS^2}{K^4 \|A\|^2} \Big).
  $$
\end{corollary}

\begin{remark}
  Informally, Corollary~\ref{cor: SBP} states that
  the small ball probability decays exponentially in the stable
  rank $r(A) = \|A\|_\HS^2 / \|A\|^2$.
\end{remark}

\begin{proof}
Let $X'$ denote an independent copy of the random vector $X$. Denote
$p = \Pr{ \|AX-y\|_2 < \frac{1}{2} \|A\|_\HS }$. Using independence and triangle inequality,
we have
\begin{align}
p^2
&= \Pr{ \|AX-y\|_2 < \frac{1}{2} \|A\|_\HS, \; \|AX'-y\|_2 < \frac{1}{2} \|A\|_\HS }  \nonumber\\
&\le \Pr{ \|A(X-X')\|_2 < \|A\|_\HS }.				\label{eq: symmetrized}
\end{align}
The components of the random vector $X-X'$ have mean zero,
variances bounded below by $2$ and sub-gaussian norms bounded above by $2K$.
Thus we can apply Theorem~\ref{thm: conc} for $\frac{1}{\sqrt{2}}(X-X')$
and conclude that
$$
\Pr{ \|A(X-X')\|_2 < \sqrt{2} (\|A\|_\HS - t) }
\le 2 \exp \Big( - \frac{c t^2}{K^4 \|A\|^2} \Big), \quad t \ge 0.
$$
Using this with $t=(1 - 1/\sqrt{2}) \|A\|_\HS$, we obtain the desired
bound for \eqref{eq: symmetrized}.
\end{proof}

The following consequence of Corollary~\ref{cor: SBP} is even more informative.
It states that $\|AX-y\|_2 \gtrsim \|A\|_\HS + \|y\|_2$ with high probability.

\begin{corollary}[Small ball probabilities, improved]		\label{cor: SBP improved}
  Let $A$ be a fixed $m \times n$ matrix.
  Consider a random vector $X = (X_1,\ldots,X_n)$ where $X_i$ are
  independent random variables satisfying
  $\E X_i = 0$, $\E X_i^2 = 1$ and $\|X_i\|_{\psi_2} \le K$.
  Then for every $y \in \R^m$ we have
  $$
  \Pr{ \|AX-y\|_2 < \frac{1}{6} (\|A\|_\HS + \|y\|_2) }
  \le 2 \exp \Big( - \frac{c \|A\|_\HS^2}{K^4 \|A\|^2} \Big).
  $$
\end{corollary}

\begin{proof}
Denote $h := \|A\|_\HS$.
Combining the conclusions of Theorem~\ref{thm: conc}
and Corollary~\ref{cor: SBP}, we obtain that with probability at least $1 - 4 \exp(-ch^2/K^4 \|A\|^2)$,
the following two estimates hold simultaneously:
\begin{equation}				\label{eq: AX AX-y}
\|AX\|_2 \le \frac{3}{2} h \quad \text{and} \quad \|AX-y\|_2 \ge \frac{1}{2} h.
\end{equation}
Suppose this event occurs. Then by triangle inequality,
$\|AX-y\|_2 \ge \|y\|_2 - \|AX\|_2 \ge \|y\|_2 - \frac{3}{2} h$.
Combining this with the second inequality in \eqref{eq: AX AX-y}, we obtain that
$$
\|AX-y\|_2 \ge \max \Big( \frac{1}{2} h, \|y\|_2 - \frac{3}{2} h \Big)
\ge \frac{1}{6}(h + \|y\|_2).
$$
The proof is complete.
\end{proof}

\section{Two applications}					\label{s: applications}

Concentration results like Theorem~\ref{thm: conc} have many useful consequences.
We include two applications in this article; the reader will certainly find more.

The first application is a concentration of distance
from a random vector to a fixed subspace.
For random vectors with bounded components,
one can find a similar result in
\cite[Corollary~2.1.19]{Tao RMT}, where it was deduced
from Talagrand's concentration inequality.

\begin{corollary}[Distance between a random vector and a subspace]
  Let $E$ be a subspace of $\R^n$ of dimension $d$.
  Consider a random vector $X = (X_1,\ldots,X_n)$ where $X_i$ are
  independent random variables satisfying
  $\E X_i = 0$, $\E X_i^2 = 1$ and $\|X_i\|_{\psi_2} \le K$.
  Then for any $t \ge 0$, we have
  $$
  \Pr{ \big| d(X,E) - \sqrt{n-d} \big| > t }
  \le 2 \exp(-c t^2 / K^4).
  $$
\end{corollary}

\begin{proof}
The conclusion follows from Theorem~\ref{thm: conc} for $A=P_{E^\perp}$,
the orthogonal projection onto $E$. Indeed,
$d(X,E) = \|P_{E^\perp} X\|_2$, $\|P_{E^\perp}\|_\HS = \dim(E^\perp) = \sqrt{n-d}$
and $\|P_{E^\perp}\| = 1$.
\end{proof}

\medskip

Our second application of Theorem~\ref{thm: conc} is for operator norms of random matrices.
The result essentially states that an $m \times n$ matrix $BG$ obtained as
a product of a deterministic matrix $B$ and a random matrix $G$ with
independent sub-gaussian entries satisfies
$$
\|BG\| \lesssim \|B\|_\HS + \sqrt{n} \|B\|
$$
with high probability.
For random matrices with heavy-tailed rather than sub-gaussian
components, this problem was studied in \cite{V products}.

\begin{theorem}[Norms of random matrices]			\label{thm: product norm}
  Let $B$ be a fixed $m \times N$ matrix, and let $G$ be an $N \times n$
  random matrix with independent entries that satisfy
  $\E G_{ij} = 0$, $\E G_{ij}^2 = 1$ and $\|G_{ij}\|_{\psi_2} \le K$.
  Then for any $s, t \ge 1$ we have
  $$
  \Pr{ \|BG\| > C K^2 (s \|B\|_\HS + t \sqrt{n} \|B\|) } \le 2 \exp(- s^2 r - t^2 n).
  $$
  Here $r = \|B\|_\HS^2 / \|B\|^2$ is the stable rank of $B$.
\end{theorem}

\begin{proof}
We need to bound $\|BGx\|_2$ uniformly for all $x \in S^{n-1}$.
Let us first fix $x \in S^{n-1}$.
By concatenating the rows of $G$, we can view $G$ as a long vector in $\R^{Nn}$.
Consider the linear operator $T : \ell_2^{Nn} \to \ell_2^m$ defined as
$T(G) = BGx$, and let us apply Theorem~\ref{thm: conc} for $T(G)$.
To this end, it is not difficult to see that the the Hilbert-Schmidt norm of $T$ equals $\|B\|_\HS$
and the operator norm of $T$ is at most $\|B\|$.
(The latter follows from $\|BGx\| \le \|B\| \|G\| \|x\|_2 \le \|B\| \|G\|_\HS$,
and from the fact the $\|G\|_\HS$ is the Euclidean norm of $G$ as a vector in $\ell_2^{Nn}$).
Then for any $u \ge 0$, we have
$$
\Pr{ \|BGx\|_2 > \|B\|_\HS + u } \le 2 \exp \Big( - \frac{c u^2}{K^4 \|B\|^2} \Big).
$$

The last part of the proof is a standard covering argument. Let $\NN$ be an $1/2$-net of
$S^{n-1}$ in the Euclidean metric. We can choose this net so that $|\NN| \le 5^n$,
see \cite[Lemma~5.2]{V RMT}. By a union bound, with probability at least
\begin{equation}				\label{eq: union bound prob}
5^n \cdot 2 \exp \Big( - \frac{c u^2}{K^4 \|B\|^2} \Big),
\end{equation}
every $x \in \NN$ satisfies $\|BGx\|_2 \le \|B\|_\HS + u$.
On this event, the approximation lemma (see \cite[Lemma~5.2]{V RMT}) implies that
every $x \in S^{n-1}$ satisfies $\|BGx\|_2 \le 2(\|B\|_\HS + u)$.
It remains to choose $u = C K^2 (s\|B\|_\HS + t \sqrt{n} \|B\|)$ with sufficiently large
absolutely constant $C$ in order to make the probability bound \eqref{eq: union bound prob}
smaller than $2 \exp(- s^2 r - t^2 n)$. This completes the proof.
\end{proof}

\begin{remark}					\label{rem: product norm special cases}
A couple of special cases in Theorem~\ref{thm: product norm} are worth mentioning.
If $B=P$ is a projection in $\R^N$ of rank $r$ then
$$
\Pr{ \|PG\| > C K^2(s \sqrt{r} + t \sqrt{n}) } \le 2 \exp(- s^2 r - t^2 n).
$$
The same holds if $B=P$ is an $r \times N$ matrix such that $P P^\tran = I_r$.

In particular, if $B = I_N$ we obtain
$$
\Pr{ \|G\| > C K^2 (s \sqrt{N} + t \sqrt{n}) } \le 2 \exp(- s^2 N - t^2 n).
$$
\end{remark}

\subsection{Complexification}
  We formulated the results in Sections~\ref{s: subgauss} and \ref{s: applications}
  for real matrices and real valued random variables.
  Using a standard complexification trick, one can easily obtain complex versions of these results.
  Let us show how to complexify Theorem~\ref{thm: conc};
  the other applications follow from it.

  Suppose $A$ is a complex matrix while $X$ is a real-valued random vector as before.
  Then we can apply Theorem~\ref{thm: conc} for the real
  $2m \times n$ matrix $\tilde{A} := \begin{bmatrix} \re A \\ \im A \end{bmatrix}$.
  Note that $\|\tilde{A}X\|_2 = \|AX\|_2$, $\|\tilde{A}\| = \|A\|$ and $\|\tilde{A}\|_\HS = \|A\|_\HS$.
  Then the conclusion of Theorem~\ref{thm: conc} follows for $A$.

  Suppose now that both $A$ and $X$ are complex.
  Let us assume that the components $X_i$
  have independent real and imaginary parts, such that
  $$
  \re X_i = 0, \quad \E(\re X_i)^2 = \frac{1}{2}, \quad \|\re X_i\|_{\psi_2} \le K,
  $$
  and similarly for $\im X_i$.
  Then we can apply Theorem~\ref{thm: conc} for the real
  $2m \times 2n$ matrix $A' := \begin{bmatrix} \re A & -\im A \\ \im A & \re A \end{bmatrix}$
  and vector $X' = \sqrt{2} \, (\re X \; \im X) \in \R^{2n}$.
  Note that $\|A'X'\|_2 = \sqrt{2} \|AX\|_2$, $\|A'\| = \|A\|$ and $\|A'\|_\HS = \sqrt{2}\|A\|_\HS$.
  Then the conclusion of Theorem~\ref{thm: conc} follows for $A$.

\end{document}